\documentclass{article}

\usepackage{amsmath,amssymb,amsthm}

\usepackage[T1]{fontenc}
\usepackage{kpfonts, baskervald}

\newtheorem{thm}{Theorem}

\newtheorem{prop}[thm]{Proposition}

\theoremstyle{definition}
\newtheorem{defn}[thm]{Definition}
\theoremstyle{remark}
\newtheorem{rmk}[thm]{Remark}
\newtheorem{exam}[thm]{Example}

\usepackage{microtype}

\usepackage[all]{xy}

\usepackage{todonotes}

\usepackage{hyperref}
\hypersetup{
  colorlinks   = true, 
  urlcolor     = blue, 
  linkcolor    = blue, 
  citecolor   = red 
}

\newcommand{\co}{\colon\thinspace}
\newcommand{\mb}[1]{\mathbb{#1}}
\newcommand{\mf}[1]{\mathfrak{#1}}

\newcommand{\op}{{op}}

\newcommand{\too}{\xrightarrow}

\newcommand{\Pres}{\mathcal{P}}

\newcommand{\cC}{\mathcal{C}}
\newcommand{\dD}{\mathcal{D}}

\DeclareMathOperator{\GL}{GL}
\DeclareMathOperator{\Ext}{Ext}
\DeclareMathOperator{\Hom}{Hom}
\DeclareMathOperator{\Map}{Map}

\DeclareMathOperator{\End}{End}
\DeclareMathOperator{\Aut}{Aut}

\DeclareMathOperator{\Fun}{Fun}

\DeclareMathOperator{\Sq}{Sq}
\DeclareMathOperator{\Pic}{Pic}
\DeclareMathOperator{\Br}{Br}

\DeclareMathOperator{\Alg}{Alg}
\DeclareMathOperator{\Cat}{Cat}
\DeclareMathOperator{\CAlg}{CAlg}
\DeclareMathOperator{\LMod}{LMod}
\DeclareMathOperator{\RMod}{RMod}
\DeclareMathOperator{\Mod}{Mod}

\DeclareMathOperator{\pic}{\mf{pic}}
\DeclareMathOperator{\br}{\mf{br}}
\DeclareMathOperator{\ZPic}{{\mb P}ic}
\DeclareMathOperator{\gl}{\mf{gl}}

\title{Adjoining roots in homotopy theory}
\author{Tyler Lawson}

\begin{document}
\maketitle

\begin{abstract}
  We use a ``twisted group algebra'' method to constructively adjoin
  formal radicals $\sqrt[n]{\alpha}$, for $\alpha$ a unit in a
  commutative ring spectrum or an invertible object in a symmetric
  monoidal $\infty$-category. We show that this construction is
  classified by maps from Eilenberg--Mac Lane objects to the unit
  spectrum $\gl_1$, the Picard spectrum $\pic$, and the Brauer
  spectrum $\br$.
\end{abstract}

\section{Introduction}

Given a commutative ring $R$ and an element $\alpha \in R$, we can
adjoin a formal radical $\sqrt[n]{\alpha}$ to $R$ by embedding $R$
into the extension ring $R[x]/(x^n - \alpha)$. These ring extensions
come equipped with a ready-made basis $\{1,x,\dots,x^{n-1}\}$ over
$R$ and are fundamental constructions in algebra. In the derived
setting, however, it is less clear when these types of constructions
are possible. Given a commutative ring spectrum $R$ and an element in
$\alpha \in \pi_0 R$, one can construct a commutative $R$-algebra with
an $n$'th root of $\alpha$ using the same type of presentation in
terms of generators and relations. However, away from characteristic
zero the universal property enjoyed by this construction is not as
strong and its coefficient ring can be unpredictable.

We can try instead to lift the algebra directly by constructing a
commutative $R$-algebra $S$ with a map $R \to S$ such that, on
coefficient rings, we have the algebraic extension:
$\pi_* S \cong (\pi_* R)[x] / (x^n - \alpha)$. Depending on $\alpha$,
such an extension may not be possible or may not be unique. It is
always possible to adjoin roots of $1$, because those algebras are
realized by the group algebras $R[C_n]$ of finite cyclic groups. If both
$\alpha$ and $n$ are units in $\pi_0 R$ then the resulting extension
on coefficient rings is \emph{\'etale}, and the obstruction theory of
Robinson \cite{robinson-gammahomology} or Goerss--Hopkins
\cite{goerss-hopkins-summary} can be used to show that the extension
ring $S$ exists, is essentially unique, and has a universal property
among $R$-algebras with such a root adjoined. If $2$ is invertible,
this means that we can adjoint $\sqrt{-1}$.  However,
Schw\"anzl--Vogt--Waldhausen showed in
\cite{schwanzl-waldhausen-vogt-roots}, using topological Hochschild
homology, that it is impossible to adjoint a square root of $-1$ to
the sphere spectrum. A different argument of Hopkins with $K(1)$-local
power operations shows that the $p$-complete $K$-theory spectrum
cannot admit $p$'th roots of unity \cite[A.6.iii]{tmforientation}, and
this was further generalized by Devalapurkar to $K(n)$-local theory
\cite{devalapurkar-roots}. Further difficulties appear when attempting
to adjoin a root that appears in a nonzero degree.

Our starting observation is that these formal radicals are a special
case of twisted group algebras. Given an abelian group extension
$0 \to G \to E \to A \to 0$ and a group homomorphism $G \to R^\times$,
we can form the relative tensor product
$\mb Z[E] \otimes_{\mb Z[G]} R$. The elements of $A$ lift to a basis
over $R$, and the resulting algebra differs from the group algebra
$R[A]$ by this central extension. Moveover, all such extensions can be
constructed in a universal case by pushing forward the extension
from $G$ to $R^\times$.

This process can be applied to ring spectra. We will begin by showing
that, when it is possible to construct similar extensions of the
spectrum of units $\gl_1(R)$, we can give systematic constructions of
formal radicals and other twisted group algebras, and obstructions are
detectable with sufficient knowledge of $\gl_1(R)$. By then
reinterpreting these constructions as Thom spectra for maps to
$\pic(R)$, we can use recent work of Antol\'in-Camarena--Barthel
\cite{antolin-camarena-barthel-thom} to both generalize this and allow
us to identify these algebras as having a universal property. This
recovers several constructions: adjoining $n$'th roots of elements to
a ring spectrum where $n$ is invertible, usually carried out using
obstruction thery, and adjoining similar roots to elements in gradings
outside zero. There are also new constructions: we find that we can
extend the first Postnikov stage $\tau_{\leq 1} \mb S_{(2)}$ of the
$2$-local sphere by adjoining $\sqrt{D}$ for $D \equiv 1$ mod $4$.

Once we have accomplished this, our second goal will be to dig one
categorical level down.

The same methods can be applied to adjoin formal radicals of elements
in the Picard group. This allows us to take an extension of the Picard
spectrum $\pic(R)$ and use it to embed the category of $R$-modules
into a graded category with a larger group of invertible objects. This
formalism recovers algebraic examples, such as Rezk's $\omega$-twisted
tensor product for $\mb Z/2$-graded modules
\cite{rezk-wilkerson}. There are also new topological examples: if $R$
is an $MU$-algebra we can embed the category $\LMod_{R}$ of
$R$-modules, where integer suspensions are possible, into a larger
category $\prod_{\mb Q/Z} \LMod_{R}$ with a symmetric monoidal
structure that allows suspensions by elements of
$\mb Q \times \mb Z/2$. This is also possible for modules over the
topological $K$-theory spectrum $ku$ and the algebraic $K$-theory
spectrum $K\mb C$. Although our focus is on ring spectra, many of the
results are proved in the generality of presentable symmetric monoidal
$\infty$-categories.

We finally close with the observation that algebraic $K$-theory
relates these two processes.

\subsection*{Further directions}
  
A first issue is that our discussion of adjoining roots is less
satisfying for units outside degree zero. In particular, the
identification of such units is somewhat roundabout. Ideally a
solution to this problem would make use of a spectrum of graded units
similar to those developed by Sagave \cite{sagave-gradedunits}, and in
particular his construction of $\mathrm{bgl}_1^* R$.

Second, we restrict our attention to strictly commutative objects
(meaning $E_\infty$-ring objects). The constructions in this paper
should have interesting and useful $E_n$-variants, making use of the
iterated classifying spaces $B^{(n)} GL_1 R$.

Finally, we study unit groups because they are more easily analyzed via
their associated spectrum. This means that we lose any ability to
extract formal radicals of nonunit elements. We are hopeful that the
future will bring a better understanding of the structure theory of
$E_\infty$-spaces, allowing us to move beyond unit groups to
effectively study multiplicative monoids.

\subsection*{Conventions and background}

Our paper is written homotopically, and in particular we use the
phrase ``commutative ring spectrum'' to mean an $E_\infty$-ring
spectrum.

For expedience, we will use the same name for both an abelian group
$A$ and the associated Eilenberg--Mac Lane spectrum, regarding the
category of abelian groups as embedded fully faithfully into the
higher category of spectra. A commutative ring $k$ gives rise to a
commutative ring spectrum.

For a ring spectrum $R$, the unit group
$\GL_1(R) \subset \Omega^\infty R$ is the space of units under the
multiplicative monoidal product, or equivalently the space of
self-equivalences of $R$ as a left $R$-module. If $R$ has a
commutative ring structure we write $\gl_1(R)$ for the associated
spectrum of units \cite{may-quinn-ray-ringspectra,
  abghr-units-orientations}. There is a unit map
$\mb S[\GL_1(R)] \to R$ for the adjunction between unit groups and
spherical group algebras.

For a monoidal $\infty$-category $\cC$, the Picard space
$\Pic(\cC) \subset \cC^\simeq$ is the space of invertible objects and
equivalences between them \cite{clausen-jhomomorphism,
  mathew-stojanoska-pictmf}.\footnote{(cf. \cite[\S
  2.2]{mathew-galois}) If the unit of $\cC$ is $\kappa$-compact for
  some cardinal $\kappa$, then the objects of $\Pic(\cC)$ are also
  $\kappa$-compact, and if $\cC$ is presentable then the full
  subcategory of $\kappa$-compact objects is essentially
  small. Therefore, for presentable monoidal $\infty$-categories it is
  possible to identify $\Pic(\cC)$ with a small space even though
  $\cC$ is large.} If $\cC$ has a symmetric monoidal structure then
$\Pic(\cC)$ has an $E_\infty$-structure and we write $\pic(\cC)$ for
the associated Picard spectrum. If $\cC = \LMod_R$ is the category of
modules over $R$, we simply write $\Pic(R)$ and $\pic(R)$ instead.

We will require known identifications of the groups $[A,\Sigma^k B]$
of homotopy classes of maps between Eilenberg--Mac Lane spectra, which
we will simply state
\cite{eilenberg-maclane-emspaces-computation}.\footnote{All of these
  can be determined by first calculating that the groups of maps
  $\mb Z \to \Sigma^k \mb Z$ are $\mb Z, 0, 0, \mb Z/2, 0$ for
  $k = 0\dots 5$, and then using free resolutions of the source and
  target. The generator in degree $3$ is the composite of mod-$2$
  reduction, the Steenrod square $\Sq^2$, and the Bockstein.}
\begin{enumerate}
\item The group $[A,B]$ is isomorphic to the group $\Hom(A,B)$.
\item The group $[A,\Sigma B]$ is isomorphic to the group $\Ext(A,B)$:
  this extension is identified with the fiber of a map $A \to \Sigma B$.
\item The group $[A,\Sigma^2 B]$ is isomorphic to the group
  $\Hom(A, B[2])$ of $2$-torsion homomorphisms $A \to B$.
\item The group $[A,\Sigma^3 B]$ is part of a short exact sequence
  \[
    0 \to \Ext(A, B[2]) \to [A,\Sigma^3 B] \to \Hom(A, B/2) \to 0.
  \]
\item These identifications respect composition. In particular, the
  composition map
  $[B,\Sigma^2 C] \times [A, \Sigma B] \to [A, \Sigma^3 C]$ is the
  Yoneda pairing
  \[
    \Hom(B,C[2]) \times \Ext(A,B) \to \Ext(A,C[2]).
  \]
  As a result, a pair $(g,\Gamma)$ representing a homomorphism
  $B \to C[2]$ and an extension $0 \to B \to \Gamma \to A \to 0$ maps
  to zero if and only if the homomorphism $g$ extends from $B$ to all
  of $\Gamma$.
\end{enumerate}

Given an ordinary symmetric monoidal category $\dD$, the Picard space
$\Pic(\dD)$ is the nerve of an ordinary groupoid, consisting of the
invertible objects and isomorphisms between them. This means there is
a fiber sequence
\[
  \pic(\dD) \to \pi_0 \pic(\dD) \too{k} \Sigma^2 \pi_1
  \pic(\dD).
\]
Here $\pi_0 \pic(\dD)$ is the classical Picard group of the category
$\dD$, and $\pi_1 \pic(\dD)$ is the automorphism group
$\Aut_{\dD}(\mb I)$ of the monoidal unit.  This first $k$-invariant is
always expressed in terms of the twist map. Namely, given an
invertible module $\gamma$, the twist-self-isomorphism
$\gamma \otimes \gamma \to \gamma \otimes \gamma$ is multiplication by
a 2-torsion automorphism $\tau(\gamma)$ of the monoidal unit, an
element of $\pi_1 \pic(\dD)$ that satisfies
$\tau(\gamma) \circ \tau(\gamma) = {\rm id}$. The $k$-invariant
$\pi_0 \pic(\dD) \too{k} \Sigma^2 \pi_1 \pic(\dD)$ is identified with
$\tau$.

\subsection*{Acknowledgements}

The author would like to thank
Tobias Barthel,
Clark Barwick,
Sanath Devalapurkar,
Fabien Hebestreit,
Lars Hesselholt,
and
Charles Rezk
for discussions related to this paper, and
Benjamin Antieau
for comments on an earlier draft. The author would also like to thank
the Max Planck Institute for their hospitality and financial support
while this paper was being written. The author was partially supported
by NSF grant 1610408.

\section{Formal radicals}

Let $\alpha \in \pi_0(R)$ be a unit; $\alpha$ is then also represented
by an element of $\pi_0(\gl_1(R))$. Fix a positive integer $n$. From
this, we can construct an extension of abelian groups
\[
  0 \to \pi_0 (\gl_1 R) \to E \to \mb Z/n \to 0
\]
such that the generator $1$ of $\mb Z/n$ has a chosen lift to an
element $x \in E$ with $x^n = \alpha$. As a set with an action of
$\pi_0 \gl_1(R)$,
$E \cong \coprod_{i=0}^{n-1} \pi_0 \gl_1(R) \cdot \{x^i\}$.

This extension is determined by an extension class
$\Ext^1(\mb Z/n, \pi_0 \gl_1(R))$, or equivalently by a map
\[
  \bar\rho\co \mb Z/n \to \Sigma \pi_0 \gl_1(R)
\]
between Eilenberg--Mac Lane spectra.

\begin{defn}
  A \emph{formal $n$'th root of $\alpha$} is a lift $\rho$
  of $\bar\rho$ to $\gl_1(R)$:
  \[
    \xymatrix{
      & \Sigma \gl_1(R) \ar[d] \\
      \mb Z/n \ar[ur]^{\rho} \ar[r]_-{\bar\rho} &
      \Sigma \pi_0 \gl_1(R)
    }
  \]
  We refer to the fiber of $\rho$ as the \emph{extended
    unit spectrum} $\gl_1(R,\rho)$ associated to $\rho$, and the
  associated infinite loop space as \emph{extended unit group}
  $GL_1(R,\rho)$.
\end{defn}

The extended unit spectrum is part of a fiber sequence
\[
  \gl_1(R) \to \gl_1(R,\rho) \to \mb Z/n,
\]
and hence we get a decomposition
$\GL_1(R,\rho) \cong \coprod_{i=0}^{n-1} \GL_1(R) \cdot \{x^i\}$ as
spaces with an action of $\GL_1(R)$.

\begin{rmk}
  The map $\rho$ determines $\bar\rho$: the map $\gl_1(R) \to
  \gl_1(R,\rho)$ is an isomorphism on $\pi_k$ except when $k=0$, when
  it is the inclusion $\pi_0 \gl_1(R) \to E$. Therefore, $\rho$
  determines the extension $E$ and hence determines $\alpha$ up to
  $n$'th powers.
\end{rmk}

\begin{defn}
  Suppose that $\rho$ is a formal $n$'th root of $\alpha$. Then
  the algebra obtained by \emph{adjoining this root} is the relative
  smash product
  \[
    R[\rho] = \mb S[\GL_1(R,\rho)] \mathop\otimes_{\mb S[\GL_1(R)]} R.
  \]
\end{defn}

\begin{prop}
  The coefficient ring of $R[\rho]$ is
  \[
    \pi_* R[\rho] \cong \pi_* R[x] / (x^n - \alpha).
  \]
\end{prop}

\begin{proof}
  The decomposition
  $\GL_1(R,\rho) \cong \coprod_{i=0}^{n-1} \GL_1(R) \cdot \{x^i\}$
  means that the spherical group algebra $\mb S[\GL_1(R,\rho)]$
  decomposes as $\oplus_{i=0}^{n-1} \mb S[\GL_1(R)] \cdot \{x^i\}$, a
  free left $\mb S[\GL_1(R)]$-module. Therefore, there is a simple
  K\"unneth formula that gives us an isomorphism of modules:
  \begin{align*}
    \pi_* R[\rho] &\cong \pi_* \mb S[\GL_1(R,\rho)]
                    \mathop\otimes_{\pi_* \mb S[\GL_1(R)]} \pi_* R \\
    &\cong \bigoplus_{i=0}^{n-1} \pi_* R \cdot \{x^i\}.
  \end{align*}
  Moreover, the identity $x^n = \alpha$ for the element
  $x \in E = \pi_0 \GL_1(R,\rho)$ completely determines the
  multiplication in $\pi_* R[\rho]$.
\end{proof}

\begin{rmk}
  It is clear that, other than the calculation of the structure of the
  coefficient ring, there is nothing special about the group $\mb Z/n$ in
  the above discussion. Given a map $\rho\co A \to \Sigma
  \gl_1(R)$ for some abelian group $A$, lifting an extension
  \[
    0 \to \pi_0 \gl_1(R) \to E \to A \to 0,
  \]
  there is an associated algebra $R[\rho]$ whose coefficient ring is a
  twisted central extension:
  \begin{align*}
    \pi_* R[\rho] &\cong \mb Z[E] \mathop\otimes_{\mb Z[\pi_0
                    \gl_1(R)]} \pi_* R\\
    &\cong \bigoplus_{a \in A} \pi_* R \cdot \{a\}.
  \end{align*}
  We will see similar algebras in later sections.
\end{rmk}

\begin{exam}
  Suppose that $n$ is a unit in $\pi_0 R$. Then $n$ acts invertibly on
  the homotopy groups $\pi_k \gl_1(R) \cong \pi_k R$ for $k > 0$:
  therefore the spectrum of maps from $\mb Z/n$ to the $0$-connected
  cover $\tau_{\geq 1} \gl_1(R)$ is trivial. Using the fiber sequence
  \[
    \Sigma \gl_1(R) \to \Sigma \pi_0 \gl_1(R) \to \Sigma^2 \tau_{\geq
      1} \gl_1(R),
  \]
  we find that, for any unit $\alpha$, the map $\mb Z/n \to \Sigma
  \pi_0 \gl_1(R)$ lifts essentially uniquely to a formal $n$'th
  root $\mb Z/n \to \Sigma \gl_1(R)$.
\end{exam}

\begin{exam}
  Let $K$ be the $p$-complete $K$-theory spectrum. Then it is possible
  to show that the group $[\mb Z/p, \Sigma \gl_1 K]$ is trivial, and
  thus this method does not allow us to adjoin $x = \sqrt[p]{\alpha}$
  for any nontrivial element $\alpha$ of
  $\mb Z_p^\times / (\mb Z_p^\times)^p$. To give a proof, however, we
  need to know structural properties of the multiplication on $K$. The
  straight-line proof we know uses Rezk's $K(1)$-local logarithm
  $\ell_1\co \gl_1 K \to K$ \cite{rezk-logarithmic}, together with
  nontrivial knowledge of low-degree $k$-invariants for $\gl_1 K$. In
  rough, the fact that Rezk's logarithm gives us an equivalence in
  degrees greater than $2$ implies that there is a diagram of fiber
  sequences:
  \[
    \xymatrix{
      \Sigma^4 (ku)^\wedge_p \ar[r] &
      \tau_{\geq 2} \gl_1 K \ar[r] \ar[d]&
      \gl_1 K \ar[d] \\
      & \Sigma^2 \mb Z_p \ar@{.>}[ul]
      & \mb Z_p^\times \ar@{.>}[ul]
    }
  \]
  Applying $[\mb Z/p, -]$ gives us a spectral sequence that computes
  $[\mb Z/p, \Sigma \gl_1 K]$; in the critical group the
  $k$-invariants of $\gl_1 K$ give this spectral sequence one
  nontrivial differential for $p > 2$ and two nontrivial differentials
  for $p=2$.

  However, the impossibility of adjoining these radicals can be shown
  directly using $K(1)$-local power operations, in a manner exactly
  analogous to the proof that one cannot adjoin roots of unity; in
  this form it generalizes. Let us sketch this argument.

  If we had such an $E_\infty$-ring $K$-algebra $L$, it would be
  $p$-complete and thus $K(1)$-local. Its coefficient ring
  $K_*[x] / (x^p - \alpha)$ would then have a $K(1)$-local power
  operation $\psi^p$, a ring endomorphism that agrees with the
  Frobenius mod $p$ \cite{hopkins-k1-local}. The element
  $\zeta = \psi^p (x) / x$ would then be a $p$'th root of unity. If
  $\zeta$ is not in $\mb Z_p^\times$, then $L$ is a $K(1)$-local
  $E_\infty$-ring containing a nontrivial $p$'th root of unity and
  Devalapurkar has shown this to be impossible
  \cite{devalapurkar-roots}. If $\zeta$ is in $\mb Z_p^\times$, then
  $\alpha \equiv \psi^p x = \zeta^{-1} x$ mod $p$, which contradicts
  the fact that $1, x, \dots, x^{p-1}$ are a basis of this ring mod
  $p$.
\end{exam}

\section{Strict units}

One source of formal roots is the theory of strictly commutative
elements.

\begin{defn}
  The space $\mb G_m(R)$ of \emph{strictly commutative units} of $R$
  is the space of maps $\mb Z \to \gl_1(R)$, or equivalently the space
  of $E_\infty$-maps $\mb Z \to \GL_1(R)$. The generator $1 \in \mb Z$
  induces forgetful maps $\mb G_m(R) \to \GL_1(R) \to \pi_0 \gl_1(R)$.
\end{defn}

In particular, a strictly commutative unit of $R$ has an underlying
unit in $\pi_0(R)$.

\begin{prop}
  Suppose $\alpha$ is a strictly commutative unit of $R$. Then, for
  any $n > 0$, $\alpha$ has a canonical lift to a formal $n$'th
  root.
\end{prop}

\begin{proof}
  The canonical Bockstein map $\mb Z/n \to \Sigma \mb Z$
  can be composed with the map $\mb Z \to \gl_1(R)$ classifying
  $\alpha$.
\end{proof}

\begin{exam}
  Let $R = S^{-1} (\tau_{\leq 1} \mb S)$ be the localization of
  the first Postnikov truncation of the sphere spectrum with respect
  to some set $S$ of primes (not containing $2$). Then there
  is a fiber sequence
  \[
    \gl_1(R) \to (S^{-1} \mb Z)^\times \too{k} \Sigma^2 \mb Z/2.
  \]
  This $k$-invariant corresponds to a (2-torsion) homomorphism $(S^{-1}
  \mb Z)^\times \to \mb Z/2$. This homomorphism is a classical
  calculation of orientation theory: it is the map
  \[
    n \mapsto \begin{cases}
      1 &\text{if }n \equiv +1 \mod 4,\\
      {-1} &\text{if }n \equiv -1 \mod 4.
    \end{cases}
  \]
  As a result, one can determine the homotopy groups of the space of
  strictly commuting elements, and in particular there is an exact
  sequence
  \[
    0 \to [\mb Z, \gl_1(R)] \to [\mb Z, (S^{-1} \mb Z)^\times] \to
    [\mb Z, \Sigma^2 \mb Z/2].
  \]
  We find that any unit in $S^{-1} \mb Z$ which is congruent to $1$
  mod $4$ lifts, essentially uniquely, to a strictly commutative unit
  of $R$. This allows us to construct commutative algebras such as
  $R[\sqrt{5}]$ and $R[\sqrt{-3}]$, even though these are ramified
  extensions on the level of coefficient rings.
\end{exam}

\begin{rmk}
  In the case of strictly commutative units, we obtain a second
  description of the algebra obtained by adjoining this root $\rho$. A
  strictly commutative element determines a composite map
  \[
    \mb S[\mb Z] \to \mb S[\GL_1(R)] \to R,
  \]
  and so we can construct the algebra $R[\rho]$ as $R \otimes_{\mb
    S[\mb Z]} \mb S[\tfrac{1}{n} \mb Z]$.

  This has the benefit that it readily lifts to a \emph{nonunit}
  version. If we define the \emph{strictly commutative multiplicative
    monoid} $\mb M_m(R)$ to be the space of $E_\infty$-maps
  \[
    \mb N \to M_1(R) = \Omega^\infty_{\otimes} R
  \]
  to the multiplicative monoid of $R$, then a strictly commutative
  element $\alpha$ can have an $n$'th root adjoined via the
  construction
  \[
    \mb S[\tfrac{1}{n} \mb N] \otimes_{\mb S[\mb N]} R.
  \]
\end{rmk}

\begin{rmk}
  A further property possessed by strict units is that it is possible
  to trivialize them. Using the natural augmentation $\mb S[\mb Z] \to
  \mb S$ of the spherical group algebra, any strictly commutative unit
  of $R$ with underlying unit $x \in \pi_0(R)$ gives rise to a new
  commutative ring spectrum
  \[
    R / (x-1) = \mb S \otimes_{\mb S[\mb Z]} R,
  \]
  whose underlying $R$-module is equivalent to the cofiber of the map
  $R \to R$ induced by $(x-1) \in \pi_0(R)$. This algebra has the
  following property: it is the universal commutative $R$-algebra with
  a chosen commuting diagram
  \[
    \xymatrix{
      \mb Z \ar[r] \ar[d] & \gl_1(R) \ar[d] \\
      \ast \ar[r] & \gl_1(R/(x-1)).
    }
  \]
  We will see similarly themed universal properties in later sections.
\end{rmk}

\section{Strict gradings}

The shift-by-1 in our definition of formal roots is strongly
suggestive: the suspended unit spectrum $\Sigma \gl_1(R)$ is a
connective cover of the Picard spectrum $\pic(R)$. In this section we
will begin exploring Picard-graded analogues of our constructions.

\begin{defn}
  Suppose that $A$ is an abelian group and $\cC$ is a symmetric
  monoidal $\infty$-category. The space of \emph{strict $A$-gradings
    for $\cC$} is the space of maps $A \to \pic(\cC)$, or equivalently
  the space of $E_\infty$-maps $A \to \Pic(\cC)$. There is a composite
  $A \to \pi_0 \pic(\cC)$, which we refer to as the \emph{underlying
    $A$-grading}.
\end{defn}

\begin{rmk}
  Suppose that we have any symmetric monoidal functor $A \to \cC$.
  The space $A$ could be regarded as a discrete groupoid, so its image
  lies in $\cC^\simeq$; the objects in $A$ have inverses under the
  monoidal product, so monoidality of $\rho$ implies that its image
  lies inside $\Pic(\cC)$. We will not distinguish between symmetric
  monoidal functors $A \to \cC$ and symmetric monoidal functors
  $A \to \Pic(\cC)$.
\end{rmk}

\begin{exam}
  Let $\cC$ be a symmetric monoidal $\infty$-category. The
  \emph{strict Picard space} of $\cC$, denoted by $\ZPic(\cC)$, is the
  space of strict $\mb Z$-gradings: maps $\mb Z \to \pic(\cC)$, or
  equivalently $E_\infty$-maps $\mb Z \to \Pic(\cC)$. The generator
  $1 \in \mb Z$ induces forgetful maps
  $\ZPic(\cC) \to \Pic(\cC) \to \pi_0 \Pic(\cC)$, and we refer to the
  image as the \emph{underlying object}.
\end{exam}

\begin{exam}
  The space of \emph{strict $n$-torsion objects} of $\cC$, denoted by
  $\ZPic^{[n]}(\cC)$, is the space of strict $\mb Z/n$-gradings: maps
  $\rho\co \mb Z/n \to \pic(\cC)$, or equivalently maps
  $\mb Z/n \to \Pic(\cC)$ of $E_\infty$-spaces. The cofiber sequence
  $\mb Z \too{n} \mb Z \to \mb Z/n$ of spectra gives rise to the
  following maps, where each double composite is a fiber sequence:
  \[
    \mu_n(R) \to \mb G_m(R) \too{n} \mb G_m(R) \too{\partial}
    \ZPic^{[n]}(R) \to \ZPic(R) \too{n} \ZPic(R)
  \]
  In particular, the map $\partial$ sends a strictly commutative element
  $\alpha\co \mb Z \to \gl_1(R)$ to the image
  \[
    \mb Z/n \to \Sigma \mb Z \too{\alpha} \Sigma \gl_1(R) \to \pic(R)
  \]
  of the formal $n$'th root associated to $\alpha$.
\end{exam}

\begin{rmk}
  For a commutative ring spectrum $R$, a strict $n$-torsion $R$-module
  with underlying left $R$-module $L$ has a choice of equivalence
  $L^{\otimes_R n} \to R$. If the module $L$ is equivalent to
  $R$, then the map $\mb Z/n \to \pi_0 \pic(R)$ is trivial and so the
  map lifts to a map $\mb Z/n \to \Sigma \gl_1(R)$: a formal
  $n$'th root. We can detect which root (up to $n$'th powers) by
  making a choice of an equivalence $R \to L$; this determines a
  composite equivalence
  $R \simeq R^{\otimes_R n} \to I^{\otimes_R n} \to R$, and hence a
  unit in $\pi_0(R)$.
\end{rmk}

\begin{exam}
  Suppose that $\cC$ is a symmetric monoidal stable $\infty$-category
  such that $n$ is a unit in the ring $\pi_0 \End_{\cC}(\mb I)$. Then
  for $k > 0$ there are isomorphisms
  \[
    \pi_k \pic(\cC) \cong \pi_{k+1} \Aut_{\cC}(\mb I) \cong
    \pi_{k+1} \End_{\cC}(\mb I),
  \]
  and the latter are acted on invertibly by $n$. Therefore, the fiber
  sequence
  \[
    \tau_{\geq 2} \pic(\cC) \to \pic(\cC) \to \pic(h\cC)
  \]
  induces an equivalence
  $\Map(\mb Z/n, \pic(\cC)) \to \Map(\mb Z/n, \pic(h\cC))$. In this
  case, strict $n$-torsion objects in $\cC$ are equivalent to strict
  $n$-torsion objects in the homotopy category $h\cC$.
\end{exam}

\begin{exam}
  For a local ring $k$, there is a fiber sequence
  \[
    \pic(k) \to \mb Z \to \Sigma^2 k^\times,
  \]
  where $\Sigma k$ is a generator of $\pi_0 \pic(k)$. The
  $k$-invariant is the twist permutation of $\Sigma k$, and is
  represented by the homomorphism
  $\mb Z \twoheadrightarrow \{\pm 1\} \to k^\times$. This
  $k$-invariant becomes trivial on $2\mb Z$, and so the category of
  $k$-modules has a strict $2\mb Z$-grading. The spaces $\mb G_m(k)$
  are connected but not contractible, so the $2\mb Z$-gradings are
  unique up to isomorphism but not canonical. We can make them
  canonical by choosing a $2\mb Z$-grading of $\mb Z$.\footnote{If $2
    = 0$ in $k$, this $k$-invariant is trivial and so the category of
  $k$-modules has a strict $\mb Z$-grading. Moreover, $\pic(\mb F_2)
  \simeq \mb Z$, and so this $\mb Z$-grading is canonical.}
\end{exam}

\begin{exam}[{\cite[1.3.7]{secondary}}]
  For any commutative ring spectrum $R$, the element $\Sigma R$ in
  $\Pic(R)$ determines a map
  $\mb Z = \pi_0 \pic(\mb S) \to \pi_0 \pic(R)$. For any $d > 0$, we
  get a composite $d\mb Z \to \pi_0 \pic(R)$. If this lifts to a
  strict $d\mb Z$-grading, we could call this an
  $E_\infty^d$-structure on $R$: it should be a strengthening of
  the notion of an $H_\infty^d$-structure from
  \cite{bmms-hinfty}. However, our enthusiasm for extending this
  notational convention is very low.

  The universal property of the real bordism spectrum $MO$ is that it
  is initial among commutative ring spectra with a nullhomotopy of the
  map $BO \to \Pic(\mb S) \to \Pic(MO)$ of $E_\infty$-spaces.
  Equivalently, it is initial among commutative rings with a
  commutative diagram of spectra
  \[
    \xymatrix{
      ko \ar[r] \ar[d] & \pic(\mb S) \ar[d] \\
      \mb Z \ar[r] & \pic(MO).
    }
  \]
  In particular, this gives the real bordism spectrum $MO$ a strict
  $\mb Z$-grading and any commutative $MO$-algebra inherits it.
  Similar considerations with complex or spin structures structures
  give the complex bordism spectrum $MU$ a strict $2\mb Z$-grading and
  the spin bordism spectrum $MSpin$ a strict $4\mb Z$-grading.
\end{exam}

\begin{exam}
  The Atiyah--Bott--Shapiro orientation lifts to give the complex
  $K$-theory spectrum $ku$ a commutative $MU$-algebra structure,
  and the real $K$-theory spectrum $ko$ a commutative
  $MSpin$-algebra structure, due to work of Joachim
  \cite{joachim-abs}. Therefore, $ku$ admits a strict $2\mb
  Z$-grading and $ko$ admits a strict $4\mb Z$-grading.
\end{exam}

\begin{exam}
  Let $K\mb C$ be the algebraic $K$-theory spectrum of the complex
  numbers, which comes equipped with a map $f\co K\mb C \to ku$ to the
  complex topological $K$-theory spectrum. Work of Suslin showed that
  the map $f$ is an equivalence after profinite completion, and hence
  the fiber of $f$ is rational. We would like to show that $K\mb C$
  has a strict $2\mb Z$-grading. (A similar argument applies to show
  that the strict $4\mb Z$-grading of the real $K$-theory spectrum
  $ko$ lifts to the algebraic $K$-theory $K\mb R$.)

  Consider the arithmetic square:
  \[
    \xymatrix{
      K\mb C \ar[r] \ar[d] & K\mb C_{\mb Q} \ar[d] \\
      K\mb C^\wedge \ar[r] & (K\mb C^\wedge)_{\mb Q}
    }
  \]
  The functor $GL_1$ preserves this pullback. When we apply $\pic$ we
  get a diagram of connective spectra:
  \[
    \xymatrix{
      \pic(K\mb C) \ar[r] \ar[d] & \pic (K\mb C_{\mb Q})\ar[d]\\
      \pic(K\mb C^\wedge) \ar[r] & \pic((K\mb C^\wedge)_{\mb Q})
    }
  \]
  On $\pi_0$ this is the constant square $\mb Z$, and on $\pi_1$ we get
  a bicartesian square
  \[
    \xymatrix{
      \mb Z^\times \ar[r] \ar[d] & \mb Q^\times \ar[d] \\
      \widehat{\mb Z}^\times \ar[r] & (\widehat{\mb Z}_{\mb Q})^\times.
    }
  \]
  Together these show that the diagram of Picard spectra is
  a homotopy pullback diagram. Let $C$ be the cofiber of
  $\pic(K\mb C) \to \pic(K\mb C^\wedge)$; its homotopy groups are then
  rational above degree $2$, and equal to the torsion-free group
  $\widehat{\mb Z}^\times / \mb Z^\times$ in degree one.

  The obstruction to lifting the strict $2\mb Z$-grading
  \[
    2\mb Z \to \pic(ku) \to \pic(ku^\wedge) \simeq \pic(K\mb C^\wedge)
  \]
  to a $2\mb Z$-grading of $\pic(K\mb C)$ is the map $2\mb Z \to
  C$. However, consider the fiber sequence
  \[
    \tau_{\geq 2} C \to C \to \Sigma \widehat{\mb Z}^\times/\mb
    Z^\times.
  \]
  Since the left term is 1-connected and rational and the right term
  is torsion-free, there are no nontrivial homotopy classes of maps
  from $2\mb Z$ into either term, and hence $[\mb Z,C] =
  0$. Therefore, the strict $2\mb Z$-grading of $ku$ can be extended to
  $K\mb C$.

\end{exam}

\section{Trivializing algebras}

The ring spectrum constructed by adjoining formal radicals turns out
to be a special case of a more general construction associated to
strict gradings. From this point forward we will need to make heavier
use of \cite{lurie-htt, lurie-higheralgebra}.

\begin{defn}
  Suppose that $\cC$ is a presentable symmetric monoidal
  $\infty$-category, $A$ is an abelian group regarded as a discrete
  symmetric monoidal category, and that $\rho\co A \to \cC$ is a
  strict $A$-grading. We define the \emph{trivializing algebra
    $T_\rho$} to be the homotopy colimit of $\rho$.
\end{defn}

\begin{rmk}
  Since $A$ is a discrete category, there is an equivalence in $\cC$
  of the form
  \[
    T_\rho \simeq \coprod_{a \in A} \rho(a).
  \]
\end{rmk}

This homotopy colimit is a very special case of the Thom
construction. As such, work of Antol\'in-Camarena--Barthel gives the
trivializing algebra a universal property, generalizing the results of
\cite{may-quinn-ray-ringspectra, ando-blumberg-gepner-hopkins-rezk}
from the category of spectra.

\begin{prop}
  \label{prop:universalproperty}
  Suppose that $A$ is an abelian group and that $\rho\co A \to \cC$
  is a symmetric monoidal functor with trivializing algebra $T_\rho$.
  \begin{enumerate}
  \item $T_\rho$ has a natural lift to a commutative algebra object:
    $T_\rho \in \CAlg(\cC)$.
  \item The algebra $T_\rho$ is universal among commutative algebras in
    $\CAlg(\cC)$ with a chosen commuting diagram
    \[
      \xymatrix{
        A \ar[r]^-{\rho} \ar[d] & \pic(\cC) \ar[d] \\
        \ast \ar[r] & \pic(\LMod_{T_\rho}).
      }
    \]
    In particular, for any $a \in A$ the algebra $T_\rho$ has a chosen
    equivalence of left $T_\rho$-modules $\phi_a\co T_\rho \to T_\rho \otimes
    \rho(a)$, and there are chosen coherences $\phi_a \otimes 1
    \circ \phi_b \simeq \phi_{ab}$.
  \end{enumerate}
\end{prop}

\begin{proof}
  The colimit has a commutative algebra structure by \cite[Theorem
  2.8]{antolin-camarena-barthel-thom}.
  
  We will now prove the universal property essentially, following the
  same line of argument in \cite[Lemma
  3.15]{antolin-camarena-barthel-thom}. Applying \cite[Theorem
  2.13]{antolin-camarena-barthel-thom} to the functor $A \to \cC$ of
  symmetric monoidal $\infty$-categories, we find the following: maps
  $T_\rho \to R$ in $\CAlg(\cC)$ are equivalent to lax symmetric
  monoidal lifts in the diagram
  \[
    \xymatrix{
      && \cC_{/R} \ar[d] \\
      A \ar[r] \ar@{.>}[urr] & \Pic(\cC) \ar[r] & \cC.
    }
  \]
  The objects of $\cC_{/R}$ are maps $N \to R$,
  with symmetric monoidal product given by
  \[
    (N \to R) \otimes (M \to R) \simeq (N \otimes M \to R \otimes R
    \to R).
  \]
  The monoidal unit is the map $\mb I \to R$. There is a symmetric
  monoidal functor $\cC_{/R} \to \cC$, given by forgetting the
  structure map, and a symmetric monoidal functor
  $\cC_{/R} \to (\LMod_R)_{/R}$, given by
  $(L \to R) \mapsto (R \otimes L\to R)$.
  
  However, since $A$ is grouplike the image $L \to R$ of any object
  must be contained in the invertible objects of $\cC_{/R}$. This
  implies first that $L$ is an invertible object of $\cC$. This
  implies second that $R \otimes L \to R$ is an invertible object of
  $(\LMod_R)_{/R}$; this happens only when this adjoint structure map
  is an equivalence. Conversely, if $L \to R$ is a map whose adjoint
  $R \otimes L \to R$ is an equivalence, tensoring with $L^{-1}$ gives
  an equivalence of $R$-modules $R \to R \otimes L^{-1}$ of left
  $R$-modules, whose inverse is adjoint to a map $L^{-1} \to R$.
\end{proof}

\begin{rmk}
  Suppose that the functor $\rho\co A \to \cC$ maps entirely to the
  unit component: it is a map $A \to B\Aut_\cC(\mb I)$ of
  $E_\infty$-spaces, classifying an extension $G$ of the automorphism
  space $\Aut_{\cC}(\mb I)$. Then the trivializing algebra $T_\rho$ is
  universal among commutative algebras $R$ in $\cC$ where the map
  $\Aut_{\cC}(\mb I) \to \Aut_{\cC}(R)$ extends to a map of
  $E_\infty$-spaces $G \to \Aut_{\cC}(R)$.

  In particular, when $\cC$ is the category $\Mod_R$ of modules over a
  commutative ring spectrum, this universal property recovers the
  relative tensor product defined earlier.
\end{rmk}

\begin{exam}
  The trivializing algebra for the map $2\mb Z \to \pic(MU)$ is a
  periodic $MU$-spectrum $MUP$, whose coefficient ring is
  $\pi_* MU[u^{\pm 1}]$ for a generator $u$ in degree $2$. It is
  universal among commutative algebras $R$ with a nullhomotopy of the
  composite $ku \to \pic(\mb S) \to \pic(R)$. The algebra $MUP$ is
  often useful for translating between even-periodic and
  $\mb Z/2$-graded interpretations in chromatic theory.
\end{exam}

\section{Root obstructions}
\label{sec:root-obstructions}

The symmetric monoidal functor from $\cC$ to its homotopy category
$h\cC$ induces a map of Picard spectra $\pic(\cC) \to \pic(h\cC)$,
identifying $\pic(h\cC)$ with the first nontrivial stage
$\tau_{\leq 1} \pic(\cC)$ in the Postnikov tower for $\pic(\cC)$. For
us to lift a map $\bar\rho\co A \to \pi_0 \pic(\cC)$ to the first
Postnikov stage $\pic(h\cC)$, it is necessary and sufficient that the
composite map
\[
  A \too{\bar\rho} \pi_0 \pic(\cC) \too{k} \Sigma^2 \pi_1 \pic(\cC)
\]
is trivial. The result is an obstruction class: an element of
$[A, \Sigma^2 \pi_1 \pic(\cC))]$, and a lift exists if and only if
this obstruction vanishes. Two different choices of lift to a map
$A \to \pic(h\cC)$ are represented by homotopy classes of maps
$A \to \Sigma \pi_1 \pic(\cC)$.

The identification of the first $k$-invariant with the twist
homomorphism $\tau$ leads us to the following definition
\cite{rezk-wilkerson}.
\begin{defn}
  An element $\gamma \in \pi_0 \pic(\cC)$ is \emph{symmetric} if
  $\tau(\gamma) = {\rm id}$ in $\pi_1 \pic(\cC)$.
\end{defn}

This allows us to conclude the following.
\begin{prop}
  A map $\bar\rho\co A \to \pi_0 \pic(\cC)$ lifts to a map
  $A \to \pic(h\cC)$ if and only if the elements $\bar\rho(a)$ are
  symmetric for all $a \in A$. Two such lifts determine a difference
  class in $\Ext(A,\pi_1 \pic(h\cC))$.
\end{prop}

\begin{exam}
  Suppose that $L$ is an invertible object such that there is an
  equivalence $v\co \mb I \to L^{\otimes n}$. The object $L$
  determines a map $\mb S \to \pic(\cC)$, and the equivalence $v$
  determines a nullhomotopy of $n$ times this map. Alternatively,
  there is a commutative diagram of symmetric monoidal
  $\infty$-categories
  \[
    \xymatrix{
      F(L^{\otimes n}) \ar[r]  \ar[d] & F(L) \ar[d] \\
      \ast \ar[r] & \Pic(\cC),
    }
  \]
  where $F(x)$ is the free $E_\infty$-space on an object named
  $x$. Taking associated spectra gives a diagram
  \[
    \xymatrix{
      \mb S \ar[r]^n \ar[d] & \mb S \ar[d] \\
      \ast \ar[r] & \pic(\cC),
    }
  \]
  or equivalently a map $\mb S/n \to \pic(\cC)$. Conversely, there is
  a short exact sequence
  \[
    0 \to \pi_1 \pic(\cC) / [\pi_1 \pic(\cC)]^n \to [\mb S/n,
    \pic(\cC)] \to \pi_0 \pic(\cC)[n] \to 0.
  \]
  The quotient expresses that a map from $\mb S/n$ determines an
  underlying $n$-torsion object $L$ in $\pi_0 \pic(\cC)$. The kernel
  expresses that two different maps representing the same object $L$
  may differ in their choice of equivalence
  $v\co \mb I \to L^{\otimes n}$ (modulo self-equivalences of $L$).
  
  If there is a symmetric monoidal functor $f\co \cC \to \dD$ such
  that there is an equivalence $u\co \mb I \to f(L)$ in $\dD$, then
  the map $\mb S \to \pic(\dD)$ determining $f(L)$ becomes
  trivial. However, the extended map $\mb S/n \to \pic(\dD)$ does not
  always become trivial: it becomes trivial precisely when there
  exists a choice of $u\co \mb I \to f(L)$ in $\dD$ such that
  $u^n = f(v)$.

  The bottom homotopy group of $\mb S/n$ is $\mb Z/n$. The map
  $\mb S/n \to \pic(h\cC)$ extends to a map $\mb Z/n \to \pic(h\cC)$
  if and only if $L$ is symmetric, and if so it extends uniquely. A
  strict $\mb Z/n$-grading would be an extension to a map
  $\mb Z/n \to \pic(\cC)$. In rough, we can think of this in the
  following way. If we have a strict $\mb Z/n$-grading, then it is a
  rigid version of choosing an object $L$ with an equivalence
  $v\co \mb I \to \mb L^{\otimes n}$; the trivializing algebra $T$
  then extracts an $n$'th root of this \emph{chosen} equivalence $v$.
\end{exam}

\begin{exam}
  Suppose that $R$ is a commutative ring spectrum which is
  $2n$-periodic: there is a unit $v \in \pi_{2n} R$. This determines a
  symmetric element $\Sigma^2R$, and a lift of $\Sigma^2 R$ to a
  strict $n$-torsion object allows us to construct an $R$-algebra
  whose $\mb Z$-graded coefficient ring is
  \[
    \pi_* R[x] / (x^n - a v)
  \]
  for some well-defined
  $a \in (\pi_0 R)^\times / [(\pi_0 R)^\times]^n$. If $n$ is a unit in
  $\pi_0 R$ then we also find that such algebras can be constructed,
  essentially uniquely, for any value of $a$. These extensions appear,
  for example, when relating completed Johnson--Wilson
  spectra to Lubin--Tate spectra \cite[\S 4]{level3}.
\end{exam}

\section{Grading extensions}

In this section, we will begin the process of extending gradings by
adjoining formal radicals to elements in the Picard group. We fix a
symmetric monoidal presentable $\infty$-category $\cC$, and let
$\Pic(\cC)$ be the Picard space of invertible elements in $\cC$.

\begin{defn}
  Let
  \[
    0 \to \pi_0 \pic(\cC) \to \Gamma \to B \to 0
  \]
  be an extension of abelian groups, represented by a map
  \[
    \bar\rho\co B \to \Sigma \pi_0 \pic(\cC)
  \]
  between Eilenberg--Mac Lane spectra.  A \emph{extension to
    $\Gamma$-grading} is a lift of $\bar\rho$ to $\pic(\cC)$:
  \[
    \xymatrix{
      & \Sigma \pic(\cC) \ar[d] \\
      B \ar[ur]^\rho \ar[r]_-{\bar\rho} & \Sigma \pi_0 \pic(\cC)
    }
  \]
  We refer to the fiber of $\rho$ as the \emph{extended Picard
    spectrum} $\pic(\cC,\rho)$ associated to $\rho$, and the associated
  infinite loop space as the \emph{extended Picard group}
  $\Pic(\cC,\rho)$.
\end{defn}

The extended Picard spectrum is part of a fiber sequence
\[
  \pic(\cC) \to \pic(\cC,\rho) \to B,
\]
and on $\pi_0$ this realizes the extension
$0 \to \pi_0 \pic(\cC) \to \Gamma \to B \to 0$. We get a decomposition
$\Pic(\cC,\rho) \cong \coprod_{b \in B} \Pic(\cC) \cdot \{b\}$ as
spaces with an action of $\Pic(\cC)$.

\begin{rmk}
  Again, the map $\rho$ determines $\bar\rho$ and the extension
  $\Gamma$.
\end{rmk}

By definition, there is an inclusion $i\co \Pic(\cC) \subset \cC$ of
symmetric monoidal $\infty$-categories. The latter is presentable,
whereas the former is (essentially) small. By formally adjoining
colimits to $\cC$, we obtain a factorization
\[
  \Pic(\cC) \to \Pres(\Pic(\cC)) \to \cC
\]
through the presheaf $\infty$-category, where the second functor
preserves colimits.

\begin{prop}
  For a symmetric monoidal $\infty$-category $\cC$, there is a diagram
  \[
    \Pres(\Pic(\cC,\rho)) \leftarrow \Pres(\Pic(\cC)) \to \cC
  \]
  of symmetric monoidal presentable $\infty$-categories.
\end{prop}

\begin{proof}
  Fix a regular cardinal $\kappa$ such that the unit of $\cC$ is
  $\kappa$-compact. Then all of the objects of $\Pic(\cC)$ are
  contained inside the essentially small subcategory $\cC^\kappa$ of
  $\kappa$-compact objects.
  
  The functor $\Pres$ is the left adjoint in an adjunction between
  $\kappa$-small $\infty$-categories and $\kappa$-presentable
  $\infty$-categories; $(-)^\kappa$ is the right adjoint. Moreover,
  the tensor product of presentable $\infty$-categories is universal
  with respect to functors that are colimit-preserving in each
  variable separately; in particular, this gives us canonical
  identifications
  \[
    \Fun^{PrL}(\Pres(\Pi S_i), \cC) \simeq \Fun^{PrL}(\otimes
    \Pres(S_i), \cC)
  \]
  natural in $\cC$. This makes the functor $\Pres$ strong symmetric
  monoidal, and lifts it to a left adjoint to the functor taking a
  $\kappa$-presentable symmetric monoidal $\infty$-category to the
  symmetric monoidal subcategory of $\kappa$-compact objects. For a
  small symmetric monoidal $\infty$-category $S$, the induced
  symmetric monoidal structure on the category $\Pres(S)$ is given by
  left Kan extension: this is the Day convolution monoidal structure
  \cite{glasman-day, lurie-higheralgebra}. It is colimit-preserving in
  each variable, and for objects $s$ and $t$ of $S$ with associated
  presheaves $j_s$ and $j_t$ there is a natural isomorphism
  $j_s \otimes j_t \cong j_{s \otimes t}$.

  The Day convolution makes the functor
  $\Pres(\Pic(\cC)) \to \Pres(\Pic(\cC,\rho))$ symmetric monoidal, and
  the adjunction gives us a composite symmetric monoidal functor
  \[
    \Pres(\Pic(\cC)) \to \Pres(\cC^\kappa) \to \cC
  \]
  as desired.
\end{proof}

\begin{defn}
  Suppose that $\rho$ is an extension to $\Gamma$-grading. We define
  the category obtained by \emph{extending gradings to $\Gamma$} to be
  the symmetric presentable $\infty$-category
  \[
    \cC[\rho] = \Pres(\Pic(\cC,\rho)) \otimes_{\Pres(\Pic(R))} \cC.
  \]
\end{defn}

\begin{prop}
  As a presentable category left-tensored over $\cC$, we have
  \[
    \cC[\rho] \cong \prod_{b \in B} \cC.
  \]
  In particular, $\cC[\rho]$ is isomorphic to the category of
  $B$-graded objects of $\cC$.
\end{prop}

\begin{proof}
  Since $\Pic(\cC,\rho) \cong \coprod_{b \in B} \Pic(\cC)$ as categories
  left-tensored over $\Pic(\cC)$,
  \[
    \Pres(\Pic(\cC,\rho)) \cong \coprod^{pres}_{b \in B}
    \Pres(\Pic(\cC))
  \]
  as presentable $\infty$-categories left-tensored over
  $\Pres(\cC)$---here the coproduct taking place within
  presentable $\infty$-categories. The relative tensor product
  preserves colimits in each variable, and thus we have
  \[
    \cC[\rho] \simeq \coprod^{pres}_{b \in B} \cC
  \]
  as categories left-tensored over $\cC$. However, within presentable
  $\infty$-categories, coproducts and products over a small index set
  coincide.
\end{proof}

One source of grading extensions is the theory of strict gradings.

\begin{prop}
  Suppose that $0 \to G \to \Gamma \to B \to 0$ is an extension of
  abelian groups. Then every strict $G$-grading of $\cC$ has a
  naturally associated extension to a $\Gamma$-grading.
\end{prop}

\begin{proof}
  The extension $\Gamma$ is classified by a map $B \to \Sigma G$,
  which can be composed with the strict $G$-grading $G \to \pic(\cC)$. 
\end{proof}

\begin{exam}
  Since $MO$ has a strict $\mb Z$-grading, for any $MO$-algebra $R$ we
  can then adjoin invertible objects $L$ to the category of
  $MO$-modules such that $L^{\otimes n} \simeq \Sigma MO$, or extend
  to any grading $\mb Z \subset \Gamma$. For example, we can embed
  the category of $MO$-modules into the category of $\mb Q/\mb Z$-graded
  $MO$-modules, giving the latter a symmetric monoidal structure
  where shifts by integers are extended to shifts by rational numbers.
\end{exam}

\begin{exam}
  Similarly, the strict $2\mb Z$-grading on $MU$ allow us to adjoin
  invertible objects $L$ such that $L^{\otimes n} \cong \Sigma^2
  MU$. (Note that if we take $n=2$ in this construction, we find that
  the object $\Sigma^{-1} L$ is a nontrivial object with
  $(\Sigma^{-1} L)^{\otimes 2} \cong MU$.) This allows us to extend
  from a $\mb Z$-grading on $MU$-modules to a grading over
  $\mb Z \oplus_{2\mb Z} \mb Q \cong \mb Q \times \mb Z/2$. Similar
  constructions are possible with $MU$-algebras like $ku$ or with the
  algebraic $K$-theory spectrum $K\mb C$.
\end{exam}
\section{Grading obstructions}
\label{sec:grading-obstructions}

As in \S\ref{sec:root-obstructions}, we can identify $\pic(h\cC)$ with
the first nontrivial stage $\tau_{\leq 1} \pic(\cC)$ in the Postnikov
tower for $\pic(\cC)$. For us to lift a map
$\bar\rho\co B \to \Sigma \pi_0 \pic(\cC)$ to the first Postnikov
stage, it is necessary and sufficient that the associated obstruction
\[
  B \too{\bar\rho} \Sigma \pi_0 \pic(\cC) \too{k} \Sigma^3 \pi_1 \pic(\cC)
\]
is trivial. Two different choices of lift to a map $B \to \Sigma \pic(h\cC)$
differ by an element of $[B \to \Sigma^2 \pi_1 \pic(\cC)]$.

This can be concisely packaged into the following result.
\begin{prop}
  Given an extension $0 \to \pi_0 \pic(\cC) \to \Gamma \to B \to 0$,
  the lifts of the map $\bar\rho\co B \to \Sigma \pi_0 \pic(\cC)$
  to a map $\rho_{\leq 1}\co B \to \Sigma \pic(h\cC)$ are in bijective
  correspondence with extensions of the twist homomorphism
  $\tau \co \pi_0 \pic(\cC) \to \pi_1 \pic(\cC)[2]$ to all of
  $\Gamma$.
\end{prop}

\begin{exam}
  This construction can recover the twisted $\mb Z/2$-graded
  categories of Rezk \cite[\S 2]{rezk-wilkerson}.\footnote{Rezk gives
    these constructions in the case where $\cC$ is merely additive,
    which is not covered by our assumption that $\cC$ is presentable.}
  Let $\cC$ be an ordinary presentable symmetric monoidal category and
  $\omega$ an invertible object of $\cC$. We would like to construct a
  larger category $\cC[\sqrt{\omega}]$ where $\omega$ is the square of
  another invertible module. This would be a symmetric monoidal
  category of $\mb Z/2$-graded objects $(A_0,A_1)$ of $\cC$,
  representing $A_0 \oplus (\sqrt{\omega} \otimes A_1)$, with a tensor
  product satisfying
  \[
    (A_0,A_1) \otimes (B_0, B_1) \cong \Big((A_0 \otimes B_0) \oplus
    (\omega\otimes A_1 \otimes B_1), \\(A_0 \otimes B_1) \oplus (A_1
    \otimes B_0)\Big).
  \]
  In this case, the underlying category is ordinary and so
  $\pic(\cC) = \pic(h\cC)$. For us to extend gradings, it is necessary
  and sufficient that $\omega$ be symmetric. If this is the case, the
  possible choices of extension represent choices of 2-torsion element
  $\tau(\sqrt{\omega}) \in \Aut_{\cC}(\mb I)$, a Koszul sign rule for
  the twist isomorphism on the square root of $\omega$. If $\cC$ is
  additive, choosing $\tau(\sqrt{\omega}) = -1$ recovers Rezk's
  \emph{$\omega$-twisted tensor product}.
\end{exam}

\section{Brauer roots}

Just as the unit spectrum is extended by the Picard spectrum, the
Picard spectrum is extended by the Brauer spectrum
\cite{brauerspectra,haugseng-morita}. Fix a symmetric monoidal
presentable $\infty$-category $\cC$ and let $\Cat_\cC$ be the category
of presentable $\infty$-categories left-tensored over $\cC$. This has
a symmetric monoidal under $\otimes_{\cC}$. We write
$\Br(\cC) = \Pic(\Cat_\cC)$ for the Brauer space parametrizing
invertible objects of $\Cat_\cC$ that admit a compact generator, and
$\br(\cC)$ for the associated spectrum.

The unit of $\Cat_{\cC}$ is $\cC$ itself, and all $\cC$-linear
functors are of the form $X \mapsto X \otimes B$ for some $B \in \cC$;
in particular, this identifies the $\cC$-linear functors $\cC \to \cC$
with $\cC^{\simeq}$ itself, with composition given by the tensor in
$\cC$. As a result, the space of self-equivalences of the unit is
$\Pic(\cC)$, and so there is a fiber sequence
\[
  \Sigma \pic(\cC) \to \br(\cC) \to \pi_0 \br(\cC).
\]

\begin{defn}
  Suppose that $\cC$ is a presentable symmetric monoidal
  $\infty$-category, $B$ is an abelian group regarded as a discrete
  symmetric monoidal category, and that $\rho\co B \to \Cat(\cC)$ is a
  symmetric monoidal functor. We define the \emph{trivializing
    category} $\mathcal{T}_\rho$ to be the homotopy colimit of $\rho$,
  calculated in $\Cat_{\cC}$.
\end{defn}

\begin{rmk}
  Since $B$ is a discrete category, there is an equivalence in
  $\Cat_{\cC}$ of the form
  \[
    \mathcal{T}_\rho \simeq \coprod^{\Cat_{\cC}}_{b \in B} \rho(b)
    \simeq \prod_{b \in B} \rho(b).
  \]
\end{rmk}

\begin{prop}
  Suppose that $B$ is an abelian group and that $\rho\co B \to
  \cC$ is a symmetric monoidal functor with trivializing category
  $\mathcal{T}_\rho$.
  \begin{enumerate}
  \item The map $\rho$ has a natural lift to a map $B \to \Br(\cC)$.
  \item $\mathcal{T}_\rho$ has a natural lift to a symmetric monoidal
    $\infty$-category under $\cC$.
  \item The symmetric monoidal $\infty$-category $\mathcal{T}_\rho$ is
    universal among symmetric monoidal presentable $\infty$-categories
    under $\cC$ with a chosen commutative diagram
    \[
      \xymatrix{
        B \ar[r] \ar[d] &
        \br(\cC) \ar[d] \\
        \ast \ar[r] &
        \br(\mathcal{T}_\rho).
      }
    \]
    In particular, for any $b \in B$ there is an equivalence
    $\phi_b\co \mathcal{T}_\rho \to \mathcal{T}_\rho \otimes_{\cC}
    \rho(b)$ of presentable $\infty$-categories left-tensored over
    $\mathcal{T}_\rho$, and there are chosen equivalences of functors
    $\phi_a \otimes 1 \circ \phi_b \too{\sim} \phi_{ab}$.
  \end{enumerate}
\end{prop}

\begin{proof}
  This is Proposition~\ref{prop:universalproperty}, applied to the
  (large) $\infty$-category of presentable $\infty$-categories.
\end{proof}

\begin{exam}
  Let $\cC = \LMod_R$ where $R$ is a commutative ring spectrum. Then
  each object of $\Br(R)$ is represented by the category of right
  modules over an Azumaya $R$-algebra $Q$ which is well-defined up to
  Morita equivalence \cite{antieau-gepner-brauergroups, toen-azumaya,
    baker-richter-szymik-brauer}, see also
  \cite[5.13]{brauerspectra}. Given a symmetric monoidal functor
  $\rho\co B \to \Br(R)$, we can therefore choose algebras $Q(b)$ so
  that there is an equivalence
  \[
    \mathcal{T}_\rho \simeq \prod_{b \in B} \RMod_{Q(b)}.
  \]
  The symmetric monoidal structure on this category takes more work to
  describe. The symmetric monoidal structure on $\rho$ gives Morita
  equivalences between $Q(p) \otimes_R Q(q)$ and $Q(p+q)$, which are
  expressed by $Q(p) \otimes_R Q(q) \otimes_R Q(p+q)^\op$-modules
  $M_{p,q}$. The symmetric monoidal structure is given by a formula of
  the form
  \[
    (X_b)_{b \in B} \otimes (Y_b)_{b \in B} \cong \left(\bigoplus_{p +
        q = b} (X_p \otimes Y_q) \mathop\otimes_{Q(p) \otimes Q(q)}
      M_{p,q}\right)_{b \in B}.
  \]
  However, describing the full symmetric monoidal structure in this
  fashion would require us to carefully express coherence relations
  between tensor products of the bimodules $M_{p,q}$. We can think of
  these objects as ``coefficients for the multiplication'' that give
  the symmetric monoidal structure on this product category.
\end{exam}

\begin{rmk}
  The first $k$-invariant in the Brauer spectrum is a map
  $\pi_0 \br(\cC) \to \Sigma^2 \pi_0 \pic(\cC)$, and this is still
  expressed by the twist self-equivalence. However, now this twist
  self-equivalence occurs on the level of module categories. For a
  commutative ring spectrum $R$ with an Azumaya $R$-algebra $Q$, the
  twist equivalence of $\RMod_{Q \otimes_R Q}$ is expressed by
  tensoring with the $Q \otimes_R Q$-bimodule
  $Q \otimes_R Q^{\tau}$, where the action on the left is the standard
  one and on the right factors through the twist automorphism. The
  fact that $Q \otimes_R Q$ is Azumaya means that this bimodule must
  be of the form $Q \otimes_R Q \otimes_R \tau(Q)$ for some
  $\tau(Q) \in \Pic(R)$.\footnote{The algebraic version of this
    assignment, taking $k$-algebra automorphisms of an Azumaya algebra
    $T$ to elements in the Picard group, is part of the
    Rosenberg--Zelinsky exact sequence
    \[
      1 \to k^\times \to T^\times \to \Aut_{\Alg(k)}(T) \to
      \pi_0 \pic(k)
    \]
    that expresses the potential failure of algebra automorphisms to
    be inner. If $k$ is a field this recovers the Noether--Skolem
    theorem.}

  More concretely, we can identify $\tau(Q)$ with the $R$-module
  \[
    F_{Q\otimes_R Q\mbox{-}\text{bimod}}(Q \otimes_R Q, (Q \otimes_R Q)^\tau).
  \]
  This assignment remains relatively mysterious to the author. In
  algebraic examples it is trivial, which is shown by
  descent-theoretic methods in \cite{mo-brauer-to-pic}.
\end{rmk}

\section{Remarks on algebraic $K$-theory}

We close with some conjectural remarks on the relationship with
algebraic $K$-theory.

We recall the following from \cite[5.11]{brauerspectra}. The functor $\Mod$ is
symmetric monoidal by \cite[4.8.5.16]{lurie-higheralgebra}: for $A$
and $B$ $R$-algebras, the tensor product in $\Mod_R$-linear
$\infty$-categories has a natural equivalence
\[
  \Mod_A \otimes \Mod_B \simeq \Mod_{A \otimes_R B}.
\]
Moreover, by \cite[8.6]{gepner-groth-nikolaus-universality} the
algebraic $K$-theory functor is lax symmetric monoidal, and hence
induces a functor
\[
  \Cat_R \to \Mod_{KR}.
\]
While this functor is only lax symmetric monoidal, it does send the
unit $\Mod_R$ to the unit $KR$, and hence the restriction to the
unit component is a symmetric monoidal functor
\[
  B\Pic(R) \to BGL_1(KR).
\]
In other words, there is a spectrum map $\pic(R) \to \gl_1(KR)$,
inducing a map $\Sigma \pic(R) \to \pic(KR)$. (The author sees no
evidence that this extends to a map $\br(R) \to \pic(KR)$.)

A map $\rho\co B \to \Sigma \pic(R)$, which can be used to embed
$\Mod_R$ into a symmetric monoidal category of $B$-graded modules,
therefore also induces a map $B \to \Sigma \gl_1 KR$, which can be
used to extend the unit group of the algebraic $K$-theory
spectrum. Moreover, because $B$ is discrete the additive structure of
the extension of $\Mod_R$ and the additive structure of the extension
of $KR$ are straightforward. The universal property can be used to
shown that the algebraic $K$-theory functor takes the category of
compact $B$-graded $R$-modules, with the symmetric monoidal
structure determined by $\rho$, to the induced extension of $KR$.

\bibliography{../masterbib}

\newcommand{\etalchar}[1]{$^{#1}$}
\providecommand{\bysame}{\leavevmode\hbox to3em{\hrulefill}\thinspace}
\providecommand{\MR}{\relax\ifhmode\unskip\space\fi MR }
\providecommand{\MRhref}[2]{%
  \href{http://www.ams.org/mathscinet-getitem?mr=#1}{#2}
}
\providecommand{\href}[2]{#2}
\begin{thebibliography}{BMMS86}

\bibitem[AB14]{antolin-camarena-barthel-thom}
Omar {Antol{\'\i}n-Camarena} and Tobias {Barthel}, \emph{A simple universal
  property of {T}hom ring spectra}, arXiv e-prints (2014), arXiv:1411.7988.

\bibitem[ABG{\etalchar{+}}]{ando-blumberg-gepner-hopkins-rezk}
Matthew Ando, Andrew Blumberg, David Gepner, Michael Hopkins, and Charles Rezk,
  \emph{Units of ring spectra and thom spectra}, ar{X}iv:0810.4535.

\bibitem[ABG{\etalchar{+}}14]{abghr-units-orientations}
Matthew Ando, Andrew~J. Blumberg, David Gepner, Michael~J. Hopkins, and Charles
  Rezk, \emph{Units of ring spectra, orientations and {T}hom spectra via rigid
  infinite loop space theory}, J. Topol. \textbf{7} (2014), no.~4, 1077--1117.
  \MR{3286898}

\bibitem[AG14]{antieau-gepner-brauergroups}
Benjamin Antieau and David Gepner, \emph{Brauer groups and \'etale cohomology
  in derived algebraic geometry}, Geom. Topol. \textbf{18} (2014), no.~2,
  1149--1244. \MR{3190610}

\bibitem[BMMS86]{bmms-hinfty}
R.~R. Bruner, J.~P. May, J.~E. McClure, and M.~Steinberger, \emph{{$H_\infty $}
  ring spectra and their applications}, Lecture Notes in Mathematics, vol.
  1176, Springer-Verlag, Berlin, 1986. \MR{836132 (88e:55001)}

\bibitem[BRS]{baker-richter-szymik-brauer}
Andrew Baker, Birgit Richter, and Markus Szymik, \emph{Brauer groups for
  commutative {$S$}-algebras}, ar{X}iv:1005.5370.

\bibitem[{Cla}11]{clausen-jhomomorphism}
Dustin {Clausen}, \emph{{$p$-adic $J$-homomorphisms and a product formula}},
  arXiv e-prints (2011), arXiv:1110.5851.

\bibitem[{Dev}17]{devalapurkar-roots}
Sanath~K. {Devalapurkar}, \emph{{Roots of unity in {$K(n)$}-local
  {$E_\infty$}-rings}}, arXiv e-prints (2017), arXiv:1707.09957.

\bibitem[EML54]{eilenberg-maclane-emspaces-computation}
Samuel Eilenberg and Saunders Mac~Lane, \emph{On the groups {$H(\Pi,n)$}. {II}.
  {M}ethods of computation}, Ann. of Math. (2) \textbf{60} (1954), 49--139.
  \MR{65162}

\bibitem[GGN15]{gepner-groth-nikolaus-universality}
David Gepner, Moritz Groth, and Thomas Nikolaus, \emph{Universality of
  multiplicative infinite loop space machines}, Algebr. Geom. Topol.
  \textbf{15} (2015), no.~6, 3107--3153. \MR{3450758}

\bibitem[GH04]{goerss-hopkins-summary}
Paul~G. Goerss and Michael~J. Hopkins, \emph{Moduli spaces of commutative ring
  spectra}, Structured ring spectra, London Math. Soc. Lecture Note Ser., vol.
  315, Cambridge Univ. Press, Cambridge, 2004, pp.~151--200. \MR{2125040
  (2006b:55010)}

\bibitem[GL16]{brauerspectra}
D.~{Gepner} and T.~{Lawson}, \emph{{Brauer groups and Galois cohomology of
  commutative ring spectra}}, ArXiv e-prints (2016).

\bibitem[Gla16]{glasman-day}
Saul Glasman, \emph{Day convolution for {$\infty$}-categories}, Math. Res.
  Lett. \textbf{23} (2016), no.~5, 1369--1385. \MR{3601070}

\bibitem[Hau17]{haugseng-morita}
Rune Haugseng, \emph{The higher {M}orita category of {$\Bbb{E}_n$}-algebras},
  Geom. Topol. \textbf{21} (2017), no.~3, 1631--1730. \MR{3650080}

\bibitem[Hop14]{hopkins-k1-local}
Michael~J. Hopkins, \emph{{$K(1)$}-local {$E_\infty$}-ring spectra},
  Topological modular forms, Math. Surveys Monogr., vol. 201, Amer. Math. Soc.,
  Providence, RI, 2014, pp.~287--302. \MR{3328537}

\bibitem[Joa04]{joachim-abs}
Michael Joachim, \emph{Higher coherences for equivariant {$K$}-theory},
  Structured ring spectra, London Math. Soc. Lecture Note Ser., vol. 315,
  Cambridge Univ. Press, Cambridge, 2004, pp.~87--114. \MR{2122155}

\bibitem[Law18]{secondary}
Tyler Lawson, \emph{Secondary power operations and the {B}rown--{P}eterson
  spectrum at the prime {$2$}}, Ann. of Math. (2) \textbf{188} (2018), no.~2,
  513--576. \MR{3862946}

\bibitem[LN12]{level3}
Tyler Lawson and Niko Naumann, \emph{Commutativity conditions for truncated
  {B}rown-{P}eterson spectra of height 2}, J. Topol. \textbf{5} (2012), no.~1,
  137--168. \MR{2897051}

\bibitem[LN14]{tmforientation}
\bysame, \emph{Strictly commutative realizations of diagrams over the
  {S}teenrod algebra and topological modular forms at the prime 2}, Int. Math.
  Res. Not. IMRN (2014), no.~10, 2773--2813. \MR{3214285}

\bibitem[Lur09]{lurie-htt}
Jacob Lurie, \emph{Higher topos theory}, Annals of Mathematics Studies, vol.
  170, Princeton University Press, Princeton, NJ, 2009. \MR{2522659
  (2010j:18001)}

\bibitem[Lur17]{lurie-higheralgebra}
\bysame, \emph{Higher {A}lgebra}, Draft version available at:
  http://www.math.harvard.edu/\~{}lurie/papers/higheralgebra.pdf, 2017.

\bibitem[Mat14]{mathew-galois}
Akhil Mathew, \emph{The {G}alois group of a stable homotopy theory}, 2014.

\bibitem[May77]{may-quinn-ray-ringspectra}
J.~Peter May, \emph{{$E_{\infty }$} ring spaces and {$E_{\infty }$} ring
  spectra}, Lecture Notes in Mathematics, Vol. 577, Springer-Verlag, Berlin-New
  York, 1977, With contributions by Frank Quinn, Nigel Ray, and J{\o}rgen
  Tornehave. \MR{0494077 (58 \#13008)}

\bibitem[MS16]{mathew-stojanoska-pictmf}
Akhil Mathew and Vesna Stojanoska, \emph{The {P}icard group of topological
  modular forms via descent theory}, Geom. Topol. \textbf{20} (2016), no.~6,
  3133--3217. \MR{3590352}

\bibitem[Rez06]{rezk-logarithmic}
Charles Rezk, \emph{The units of a ring spectrum and a logarithmic cohomology
  operation}, J. Amer. Math. Soc. \textbf{19} (2006), no.~4, 969--1014.
  \MR{2219307 (2007h:55006)}

\bibitem[Rez09]{rezk-wilkerson}
\bysame, \emph{The congruence criterion for power operations in {M}orava
  {$E$}-theory}, Homology, Homotopy Appl. \textbf{11} (2009), no.~2, 327--379.
  \MR{2591924 (2011e:55021)}

\bibitem[Rob03]{robinson-gammahomology}
Alan Robinson, \emph{Gamma homology, {L}ie representations and {$E_\infty$}
  multiplications}, Invent. Math. \textbf{152} (2003), no.~2, 331--348.
  \MR{1974890}

\bibitem[Sag16]{sagave-gradedunits}
Steffen Sagave, \emph{Spectra of units for periodic ring spectra and group
  completion of graded {$E_\infty$} spaces}, Algebr. Geom. Topol. \textbf{16}
  (2016), no.~2, 1203--1251. \MR{3493419}

\bibitem[SVW99]{schwanzl-waldhausen-vogt-roots}
R.~Schw{\"a}nzl, R.~M. Vogt, and F.~Waldhausen, \emph{Adjoining roots of unity
  to {$E_\infty$} ring spectra in good cases---a remark}, Homotopy invariant
  algebraic structures ({B}altimore, {MD}, 1998), Contemp. Math., vol. 239,
  Amer. Math. Soc., Providence, RI, 1999, pp.~245--249. \MR{1718085
  (2001b:55026)}

\bibitem[To{\"e}12]{toen-azumaya}
Bertrand To{\"e}n, \emph{Derived {A}zumaya algebras and generators for twisted
  derived categories}, Invent. Math. \textbf{189} (2012), no.~3, 581--652.
  \MR{2957304}

\bibitem[Vis]{mo-brauer-to-pic}
Angelo Vistoli, \emph{On a morphism from the brauer group to the picard group},
  MathOverflow, https://mathoverflow.net/q/334293.

\end{thebibliography}
\end{document}